\documentclass[reqno]{amsart}
\usepackage{hyperref}

\begin{document}
\title[\hfilneg \hfil Uniqueness of Meromorphic Functions With Respect To Their Shifts Concerning  Derivatives]
{Uniqueness of Meromorphic Functions With Respect To Their Shifts Concerning  Derivatives}

\author[XiaoHuang Huang \hfil \hfilneg]
{XiaoHuang Huang}

\address{XiaoHuang Huang: Corresponding author\newline
Department of Mathematics, 1 Department of Mathematics, Sun Yat-sen University, Guangzhou, 510275, P.R. China}
\email{1838394005@qq.com}

\subjclass[2010]{30D35, 39A46}
\keywords{Meromorphic functions; shifts; derivatives; small functions}

\begin{abstract}
  An example in the article shows that the first derivative of $f(z)=\frac{2}{1-e^{-2z}}$ sharing $0$ CM and $1,\infty$ IM with its shift $\pi i$ cannot obtain they are equal. In this paper, we study the uniqueness of meromorphic function  sharing small functions with their shifts concerning its $k-th$ derivatives. We improves the author's result \cite{h} from entire function to meromorphic function, the first derivative to its differential-difference polynomial, and also finite values to small functions. As for $k=0$, we obtain: Let $f(z)$ be a  transcendental meromorphic function of $\rho_{2}(f)<1$, let $c$ be a nonzero finite value, and let $a_{1}(z)\not\equiv\infty, a_{2}(z)\not\equiv\infty\in \hat{S}(f)$ be two distinct small functions of  $f(z)$ such that $a(z)$ is a periodic function with period $c$ and $b(z)$ is any small function of $f(z)$. If $f(z)$ and $f(z+c)$ share $a_{1}(z),\infty$ CM, and share $a_{2}(z)$ IM, then either $f(z)\equiv f(z+c)$ or $$e^{p(z)}\equiv \frac{f(z+c)-a_{1}(z+c)}{f(z)-a_{1}(z)}\equiv \frac{a_{2}(z+c)-a_{1}(z+c)}{a_{2}(z)-a_{1}(z)},$$
where $p(z)$ is a non-constant entire function of $\rho(p)<1$ such that $e^{p(z+c)}\equiv e^{p(z)}$.
\end{abstract}

\maketitle
\numberwithin{equation}{section}
\newtheorem{theorem}{Theorem}[section]
\newtheorem{lemma}[theorem]{Lemma}
\newtheorem{remark}[theorem]{Remark}
\newtheorem{corollary}[theorem]{Corollary}
\newtheorem{example}[theorem]{Example}
\newtheorem{problem}[theorem]{Problem}
\allowdisplaybreaks

\section{Introduction and main results}

Throughout this paper, we assume that the reader have a knowledge of  the fundamental results and the standard notations of the Nevanlinna value distribution theory. See(\cite{h3,y1,y2}). In the following, a meromorphic function $f$ means meromorphic in the whole complex plane. Define
 $$\rho(f)=\varliminf_{r\rightarrow\infty}\frac{log^{+}T(r,f)}{logr},$$
 $$\rho_{2}(f)=\varlimsup_{r\rightarrow\infty}\frac{log^{+}log^{+}T(r,f)}{logr}$$
by the order  and the hyper-order  of $f$, respectively. When $\rho(f)<\infty$, we say $f$ is of finite order.

By $S(r,f)$, we denote any quantity satisfying $S(r, f) = o(T(r, f))$, as $r\to \infty $ outside of a possible exceptional set of finite logarithmic measure. A meromorphic function $a(z)$ satisfying $T(r,a)=S(r,f)$ is called a small function of $f$.  We denote $S(f)$ as the family of all small meromorphic functions of $f$ which includes the constants in $\mathbb{C}$. Moreover, we define $\hat{S}(f)=S(f)\cup\{\infty\}$. We say that two non-constant meromorphic functions $f$ and $g$ share small function $a$ CM(IM) if $f-a$ and $g-a$ have the same zeros counting multiplicities (ignoring multiplicities). Moreover, we introduce the following notation: $S_{(m,n)}(a)=\{z|z $ is a common zero of $f(z+c)-a(z)$ and $f(z)-a(z)$ with multiplicities $m$ and $n$ respectively$\}$. $\overline{N}_{(m,n)}(r,\frac{1}{f-a})$ denotes the counting function of $f$ with respect to the set $S_{(m,n)}(a)$. $\overline{N}_{n)}(r,\frac{1}{f-a})$ denotes the counting function of all distinct zeros of $f-a$ with multiplicities at most $n$. $\overline{N}_{(n}(r,\frac{1}{f-a})$ denotes the counting function of all zeros of $f-a$ with multiplicities at least $n$.

We say that two non-constant meromorphic functions $f$ and $g$ share small function $a$ CM(IM)almost if
 $$N(r,\frac{1}{f-a})+N(r,\frac{1}{g-a})-2N(r,f=a=g)=S(r,f)+S(r,g),$$
or
$$\overline{N}(r,\frac{1}{f-a})+\overline{N}(r,\frac{1}{g-a})-2\overline{N}(r,f=a=g)=S(r,f)+S(r,g),$$
respectively.

For a meromorphic function $f(z)$, we denote its shift by $f_{c}(z)=f(z + c)$.

Rubel and Yang \cite{ruy}  studied the uniqueness of an entire function concerning its first order derivative, and proved the following result.

\

{\bf Theorem A} \ Let $f(z)$ be a non-constant entire function, and let $a, b$ be two finite distinct complex values. If $f(z)$ and $f'(z)$
 share $a, b$ CM, then $f(z)\equiv f'(z)$.

Zheng and Wang \cite{zw} improved Theorem A and proved

\

{\bf Theorem B} \ Let $f(z)$ be a non-constant entire function, and let $a(z)\not\equiv\infty, b(z)\not\equiv\infty$ be two distinct small functions of $f(z)$. If $f(z)$ and $f^{(k)}(z)$ share $a(z), b(z)$ CM, then $f(z)\equiv f^{(k)}(z)$.

Li  and Yang  \cite {ly3} improved Theorem B and proved

\

{\bf Theorem C} \ Let $f(z)$ be a non-constant entire function, and let $a(z)\not\equiv\infty, b(z)\not\equiv\infty$ be two  distinct small functions of $f(z)$. If $f(z)$ and $f^{(k)}(z)$
 share $a(z)$ CM, and share $b(z)$ IM. Then $f(z)\equiv f^{(k)}(z)$.

Recently, the value distribution of meromorphic functions concerning difference analogue has become a popular research, see [1-9,  12-14, 16-18].
 Heittokangas et al \cite{hkl} obtained a similar result analogue of Theorem A concerning shifts.

\

{\bf Theorem D}
 Let $f(z)$ be a non-constant entire function of finite order, let $c$ be a nonzero finite complex value, and let $a, b$ be two finite distinct complex values.
If $f(z)$ and $f(z+c)$ share $a, b$ CM, then $f(z)\equiv f(z+c).$

In 2022, Huang\cite{h} obtained

\

{\bf Theorem E}
 Let $f(z)$ be a  transcendental entire function of finite order, let $\eta\neq0$ be a finite complex number, $n\geq1, k\geq0$  two  integers and let $a, b$ be two  distinct finite complex values. If $f(z)$ and $(\Delta_{\eta}^{n}f(z))^{(k)}$ share $a_{1}$ CM and share $a_{2}$ IM, then either $f(z)\equiv(\Delta_{\eta}^{n}f(z))^{(k)}$ or $a_{1}=2a_{2}$,
$$f(z)\equiv a_{2}e^{2(cz+d)}-2a_{2}e^{cz+d}+2a_{2},$$
and
$$(\Delta_{\eta}^{n}f(z))^{(k)}\equiv a_{2}e^{cz+d},$$
where $c=(-2)^{-\frac{n+1}{k}}$ for $k\geq1$and $d$ are two finite constants.

In the following, we define the differential polynomials of a meromorphic function. Let $f(z)$  be a  non-constant entire function  and 
\begin{align}
g(z)=b_{-1}+\sum_{i=0}^{n}b_{i}f^{(k_{i})}(z+c_{i}),
\end{align}
where $b_{-1}$ and $b_{i} (i=0\ldots,n)$ are small meromorphic functions  of $f$, $k_{i}\geq0 (i=0\ldots,n)$ are integers and $c_{i}(i=0\ldots,n)$ are finite complex numbers.

Of above theorem, it's naturally to ask whether the condition two finite complex numbers can be replaced by two distinct small functions, and $f'$ can be replaced by $g$?

In this article, we give a positive answer. In fact, we prove the following more general result.

\

{\bf Theorem 1}
Let $f(z)$ be a transcendental meromorphic function of $\rho_{2}(f)<1$ such that $\overline{N}(r,f)=S(r,f)$, let $g(z)$ be a differential polynomials of $f$ as define in (1.1), and   let $a_{1}, a_{2}$ be two distinct finite complex numbers. If $f(z)$ and $g(z)$ share $a_{1}, \infty$ CM, and  $a_{2}$ IM. Then either $f(z)\equiv g(z)$ or 
$$f(z)\equiv a_{2}+(a_{1}-a_{2})(h-1)^{2},$$
and
$$g(z)\equiv a_{1}+(a_{1}-a_{2})(h-2),$$
where $h(z)$ is a non-constant meromorphic function of $\rho_{2}(h)<1$.

\

{\bf Remark 1} When $a_{1}$ and $a_{2}$ are two distinct small functions of $f$, it is easy to see in the following proofs that Lemma 2.4 and Lemma 2.5 are still true under the assumptions that $N(r,a_{1})+N(r,a_{2})+\overline{N}(r,f)=S(r,f)$ and $f$ and $g$ share $a_{1}$ CM almost and $a_{2}$ IM almost. So we can know that Theorem 1 is still true when  $f$ and $g$ share $a_{1}$ CM almost and $a_{2}$ IM almost.

\

{\bf Corollary 1} Let $f(z)$ be a  transcendental meromorphic function of $\rho_{2}(f)<1$, let $c$ be a nonzero finite value, $k$ be a positive integer, and let $a(z)\not\equiv\infty, b(z)\not\equiv\infty\in \hat{S}(f)$ be two distinct small functions. If $f^{(k)}(z+c)$ and $f(z)$ share $a(z),\infty$ CM, and share $b(z)$ IM, then $f^{(k)}(z)\equiv f(z+c)$.

\

{\bf Example 1} \cite{hf}
Let $f(z)=\frac{2}{1-e^{-2z}}$, and let $c=\pi i$. Then $f'(z)$ and $f(z+c)$ share $0$ CM and share $1,\infty$ IM, but $f'(z)\not\equiv f(z+c)$.

This example shows that for meromorphic functions, the conclusion of Theorem 1 doesn't hold even when  sharing $\infty$ CM is replaced by sharing $\infty$ IM when $k=1$. We believe there are examples for any $k$, but we can not construct them.

As for $k=0$, Li and Yi \cite{ly} obtained

\

{\bf Theorem F} Let $f(z)$ be a  transcendental entire function of $\rho_{2}(f)<1$, let $c$ be a nonzero finite value,  and let $a(z)\not\equiv\infty, b(z)\not\equiv\infty\in \hat{S}(f)$ be two distinct small functions. If $f(z)$ and $f(z+c)$ share $a(z)$  and  $b(z)$ IM, then $f(z)\equiv f(z+c)$.

\

{\bf Remark 2} Theorem F holds when $f(z)$ is a non-constant meromorphic function of $\rho_{2}(f)<1$ such that $\overline{N}(r,f)=S(r,f)$.

Heittokangas, et. \cite{hkl1} proved.
\

{\bf Theorem G} Let $f(z)$ be a  non-constant meromorphic function of $\rho_{2}(f)<1$, let $c$ be a nonzero finite value, and let $a_{1}(z)\not\equiv\infty$, $a_{2}(z)\not\equiv\infty$  and $a_{3}(z)\not\equiv\infty\in \hat{S}(f)$ be three distinct small functions  such that $a_{1}(z)$, $a_{2}(z)$ and $a_{3}(z)$ are periodic functions with period $c$. If $f(z)$ and $f(z+c)$ share $a_{1}(z),a_{2}(z)$ CM, and $a_{3}(z)$ IM, then $f(z)\equiv f(z+c)$.

We can ask a question that whether the small periodic function $a_{3}(z)$ of $f(z)$ can be replaced by any small function of $f(z)$?\\
In this paper, we obtain our second result.

\

{\bf Theorem 2} Let $f(z)$ be a  transcendental meromorphic function of $\rho_{2}(f)<1$, let $c$ be a nonzero finite value, and let $a_{1}(z)\not\equiv\infty, a_{2}(z)\not\equiv\infty\in \hat{S}(f)$ be two distinct small functions of  $f(z)$ such that $a(z)$ is a periodic function with period $c$ and $b(z)$ is any small function of $f(z)$. If $f(z)$ and $f(z+c)$ share $a_{1}(z),\infty$ CM, and share $a_{2}(z)$ IM, then either $f(z)\equiv f(z+c)$ or $$e^{p(z)}\equiv \frac{f(z+c)-a_{1}(z+c)}{f(z)-a_{1}(z)}\equiv \frac{a_{2}(z+c)-a_{1}(z+c)}{a_{2}(z)-a_{1}(z)},$$
where $p(z)$ is a non-constant entire function of $\rho(p)<1$ such that $e^{p(z+c)}\equiv e^{p(z)}$.

We can obtain the following corollary from the proof of Theorem 2.

\

{\bf Corollary 1}  Under the same condition as in Theorem 2, then $f(z)\equiv f(z+c)$ holds if one of conditions satisfies\\
(i) $a_{2}(z)$ is a periodic function with period $c$ or $2c$;\\
(ii) $\rho(a_{2}(z))<\rho(e^{p(z)})$;\\
(iii) $\rho(a_{2}(z))<1$.

\

{\bf Example 2} \ Let $f(z)=\frac{e^{z}}{1-e^{-2z}}$, and let $c=\pi i$. Then $f(z+c)=\frac{-e^{z}}{1-e^{-2z}}$, and $f(z)$ and $f(z+c)$ share $0,\infty$ CM, but $f(z)\not\equiv f(z+c)$.

\

{\bf Example 3} \ Let $f(z)=e^{z}$, and let $c=\pi i$. Then $f(z+c)=-e^{z}$, and $f(z)$ and $f(z+c)$ share $0,\infty$ CM, $f(z)$ and $f(z+c)$ attain
different values everywhere in the complex plane, but $f(z)\not\equiv f(z+c)$.

Above two examples of  show that "2CM+1IM" is necessary.

\

{\bf Example 4} Let $f(z)=e^{e^{z}}$,  then $f(z+\pi i)=\frac{1}{e^{e^{z}}}$. It is easy to verify that $f(z)$ and $f(z+\pi i)$ share $0, 1, \infty$ CM, but $f(z)=\frac{1}{f(z+\pi i)}$. On the other hand, we  obtain $f(z)=f(z+2\pi i)$.

Example 4 tells us that if we drop the assumption  $\rho_{2}(f)<1$, we can get another relation.

By Theorem 1 and Theorem 2, we still believe the latter situation of Theorem 2 can be removed, that is to say,  only the case $f(z)\equiv f(z+c)$ occurs. So we raise a conjecture here.

\

{\bf Conjecture} Under the same condition as in Theorem 2, is  $f(z)\equiv f(z+c)$ ?

\section{Some Lemmas}
\begin{lemma}\label{1}\cite{h3} Let $f$ be a non-constant meromorphic function of $\rho_{2}(f)<1$,  and let $c$ be a non-zero complex number. Then
$$m(r,\frac{f(z+c)}{f(z)})=S(r, f),$$
for all r outside of a possible exceptional set E with finite logarithmic measure.
\end{lemma}

\begin{lemma}\label{2}\cite{h3}  Let $f$ be a non-constant meromorphic function of $\rho_{2}(f)<1$,  and let $c$ be a non-zero complex number. Then
$$T(r,f(z))=T(r,f(z+c))+S(r,f),$$
for all r outside of a possible exceptional set E with finite logarithmic measure.
\end{lemma}

\begin{lemma}\label{3}\cite{hk3,y1,y2} Let $f_{1}$ and $f_{2}$ be two non-constant meromorphic functions in $|z|<\infty$, then
$$N(r,f_{1}f_{2})-N(r,\frac{1}{f_{1}f_{2}})=N(r,f_{1})+N(r,f_{2})-N(r,\frac{1}{f_{1}})-N(r,\frac{1}{f_{2}}),$$
where $0<r<\infty$.
\end{lemma}

\begin{lemma}\label{4} Let $f(z)$ be a transcendental meromorphic function of $\rho_{2}(f)<1$ such that $\overline{N}(r,f)=S(r,f)$, let $g(z)$ be a differential polynomials of $f$ as define in (1.1), and   let $a_{1}, a_{2}$ be two distinct finite complex numbers. If
$$H=\frac{f‘}{(f-a_{1})(f-a_{2})}-\frac{g’}{(g-a_{1})(g-a_{2})}\equiv0,$$
 And  $f$ and $g$ share $a_{1}$ CM, and  $a_{2}$ IM, then either $2T(r,f)\leq\overline{N}(r,\frac{1}{f-a_{1}})+\overline{N}(r,\frac{1}{f-a_{2}})+S(r,f)$, or $f=g$.
\end{lemma}
\begin{proof}
Integrating $H$ which leads to
$$\frac{g-a_{2}}{g-a_{1}}=C\frac{f-a_{2}}{f-a_{1}},$$
where $C$ is a nonzero constant.\\

If $C=1$, then $f=g$. If $C\neq1$, then from above, we have
$$\frac{a_{1}-a_{2}}{g-a_{1}}\equiv \frac{(C-1)f-Ca_{2}+a_{1}}{f-a_{1}},$$
and
$$T(r,f)=T(r,g)+S(r,f)+S(r,g).$$
It follows that $N(r,\frac{1}{f-\frac{Ca_{2}-a_{1}}{C-1}})=N(r,\frac{1}{a_{1}-a_{2}})=S(r,f)$. Then by Lemma 2.4,
\begin{eqnarray*}
\begin{aligned}
2T(r,f)&\leq \overline{N}(r,f)+\overline{N}(r,\frac{1}{f-a_{1}})+\overline{N}(r,\frac{1}{f-a_{2}})+\overline{N}(r,\frac{1}{f-\frac{Ca_{2}-a_{1}}{C-1}})+S(r,f)\\
&\leq \overline{N}(r,\frac{1}{f-a_{1}})+\overline{N}(r,\frac{1}{f-a_{2}})+S(r,f),
\end{aligned}
\end{eqnarray*}
that is $2T(r,f)\leq \overline{N}(r,\frac{1}{f-a_{1}})+\overline{N}(r,\frac{1}{f-a_{2}})+S(r,f)$.
\end{proof}

\begin{lemma}\label{5} Let $f(z)$ be a transcendental meromorphic function of $\rho_{2}(f)<1$ such that $\overline{N}(r,f)=S(r,f)$, let $g(z)$ be a differential polynomials of $f$ as define in (1.1), and   let $a_{1}, a_{2}$ be two distinct finite complex numbers. If $f(z)$ and $g(z)$ share $a_{1}$ CM, and $N(r,\frac{1}{g(z)-(b_{-1}+b_{0}a_{1})})=S(r,f)$, then Then there are two meromorphic functions  $h$ and $H$ on $\mathbb{C}^{n}$ such that either $g=Hh+G$, where $H\not\equiv0$ and $G=b_{-1}+b_{0}a_{1}$ are two small functions of $h$, or $T(r,h)=S(r,f)$.
\end{lemma}
\begin{proof}
 Since $f(z)$ is a non-constant meromorphic function of $\rho_{2}(f)<1$, and $f(z)$ and $g(z)$ share $a_{1},\infty$ CM, then there is a meromorphic function $h$  such that
\begin{align}
f-a_{1}=h(g-G)+h(G-a_{1}),
\end{align}
where $G(z)=b_{-1}+b_{0}a_{1}$, and the zeros and poles of $h$ come from the zeros and poles of $b_{-1}$ and $b_{i}(i=0,1,\ldots,n)$

 Suppose on the contrary that $T(r,h)\neq S(r,f)$. Set $Q=g-G$. Do induction from (2.2) that
\begin{align}
Q=\sum_{i=0}^{n}b_{i}(hQ)^{(k_{i})}+\sum_{i=0}^{n}b_{i}(h(G-a_{1}))^{(k_{i})}.
\end{align}
Easy to see that $Q\not\equiv0$. Then we rewrite (2.3) as
\begin{eqnarray}
1-\frac{\sum_{i=0}^{n}b_{i}(h(G-a_{1}))^{(k_{i})}}{Q}=Fh,
\end{eqnarray}
where
\begin{align}
F&=\frac{\sum_{i=0}^{n}b_{i}(hQ)^{(k_{i})}}{hQ}
\end{align}
Note that $N(r,\frac{1}{g-G})=N(r,\frac{1}{Q})=S(r,f)$, and  hence
\begin{align}
T(r,F)&\leq\sum_{i=0}^{n}(T(r,\frac{(hQ)^{(k_{i})}}{Qh})+S(r,f)\notag\\
&\leq m(r,\frac{(hQ)^{(k_{i})}}{hQ})+N(r,\frac{(hQ)^{(k_{i})}}{hQ})+\overline{N}(r,f)\notag\\
&+S(r,h)+S(r,f)=S(r,h)+S(r,f).
\end{align}
By (2.1) and Lemma 2.1,  we get
\begin{align}
T(r,h)&\leq T(r,f)+T(r,g)+S(r,f)\notag\\
&\leq 2T(r,f)+S(r,f).
\end{align}
Then it follows from (2.5) that $T(r,F)=S(r,f)$. Next we discuss  two cases.

{\bf Case1.} \quad $h^{-1}-F\not\equiv0$. Rewrite (2.4) as
\begin{align}
h(h^{-1}-F)=\sum_{i=0}^{n}b_{i}(h(G-a_{1}))^{(k_{i})}.
\end{align}
We claim that $F\equiv0$. Otherwise, it follows from (2.9) that $N(r,\frac{1}{h^{-1}-F})=S(r,f)$. Then use Lemma 2.4 to $h$ we can obtain
\begin{align}
T(r,h)&=T(r,h^{-1})+O(1)\notag\\
&\leq \overline{N}(r,h^{-1})+\overline{N}(r,\frac{1}{h^{-1}})+\overline{N}(r,\frac{1}{h^{-1}-F})\notag\\
&+O(1)=S(r,f),
\end{align}
which contradicts with assumption. Thus $F\equiv0$. Then by (2.9) we get
\begin{align}
g=\sum_{i=0}^{n}b_{i}(h(G-a_{1})^{(k_{i})})+G=Hh+G,
\end{align}
where $H\not\equiv0$ is a small function of $h$.

{\bf Case2.} \quad $h^{-1}-F\equiv0$. Immediately, we get $T(r,h)=S(r,f)$.
\end{proof}

\begin{lemma}\label{6}\cite{y1}
Let $f$ be a non-constant meromorphic function, and $R(f)=\frac{P(f)}{Q(f)}$, where
$$P(f)=\sum_{k=0}^{p}a_{k}f^{k} \quad and \quad Q(f)=\sum_{j=0}^{q}a_{j}f^{q}$$
are two mutually prime polynomials in $f$. If the coefficients ${a_{k}}$ and ${b_{j}}$ are small functions of $f$ and $a_{p}\not\equiv0$, $b_{q}\not\equiv0$, then
$$T(r,R(f))=max\{p,q\}T(r,f)+S(r,f).$$
\end{lemma}
\begin{lemma}\label{7}\cite{y1} Let $f$ be a non-constant meromorphic function, and let $P(f)=a_{0}f^{p}+a_{1}f^{p-1}+\cdots+a_{p}(a_{0}\neq0)$ be a polynomial of degree $p$ with constant coefficients $a_{j}(j=0,1,\ldots,p)$. Suppose that $b_{j}(j=0,1,\ldots,q)(q>p)$. Then
$$m(r,\frac{P(f)f'}{(f-b_{1})(f-b_{2})\cdots(f-b_{q})})=S(r,f).$$
\end{lemma}

\begin{lemma}\label{8}\cite{y1}
Let $f$ be a non-constant meromorphic function, and let $P(f)=a_{0}+a_{1}f+a_{2}f^{2}+\cdots+a_{n}f^{n}$, where $a_{i}$ are  small functions of $f$ for $i=0,1,\ldots,n$. Then
$$T(r,P(f))=nT(r,f)+S(r,f).$$
\end{lemma}

\begin{lemma}\label{9} \cite{lh} Let $f$ and $g$ be two non-constant meromorphic functions. If $f$ and $g$ share $0,1,\infty$ IM, and $f$ is  a bilinear transformation of $g$,  then $f$ and $g$ assume one of the following six relations: (i) $fg=1$; (ii) $(f-1)(g-1)=1$; (iii) $f+g=1$; (iv) $f=cg$; (v) $f-1=c(g-1)$; (vi) $[(c-1)f+1][(c-1)g-c]=-c$, where $c\neq0,1$ is a complex number.
\end{lemma}

In the proof of Theorem 2, we will use the following Lemma proved by G. G. Gundersen \cite{g}.

\begin{lemma}\label{10}\cite{g}
 Let $f$, $F$ and $g$ be three non-constant meromorphic functions, where $g=F(f)$. Then $f$ and $g$ share three values IM if and only if there exist an entire function $h$ such that,  
 by a  suitable linear fractional transformation, one of the following cases holds: \\
 (i) $f\equiv g$;\\
 (ii) $f=e^{h}$ and $g=a(1+4ae^{-h}-4a^{2}e^{-2h})$ have three IM shared values $a\neq0$, $b=2a$ and $\infty$;\\
 (iii) $f=e^{h}$ and $g=\frac{1}{2}(e^{h}+a^{2}e^{-h})$ have three IM shared values $a\neq0$, $b=-a$ and $\infty$;\\
 (iv) $f=e^{h}$ and $g=a+b-abe^{-h}$ have three IM shared values $ab\neq0$ and $\infty$;\\  
 (v) $f=e^{h}$ and $g=\frac{1}{b}e^{2h}-2e^{h}+2b$ have three IM shared values $b\neq0$, $a=2b$ and $\infty$;\\
 (vi) $f=e^{h}$ and $g=b^{2}e^{-h}$ have three IM shared values $a\neq0$, $0$ and $\infty$.
 \end{lemma}

\begin{lemma}\label{11} \cite{hk3,y1,y2}  Let $f$ and $g$ be two non-constant meromorphic functions, and let $\rho(f)$ and  $\rho(g)$ be the order of $f$ and $g$, respectively. Then
$\rho(fg)\leq \max\{\rho(f), \rho(g)\}$.
\end{lemma}

\begin{lemma}\label{12}\cite{y3} Let $f(z)$  be a non-constant meromorphic function, and let $a_{1}, a_{2}, a_{3}\in \hat{S}(f)$ be three distinct small functions of $f$. Then
$$T(r,f)\leq \sum_{j=1}^{3}\overline{N}(r,\frac{1}{f-a_{j}})+S(r,f).$$
\end{lemma}
\

{\bf Remark 3} We can see from the proof that Lemma 2.9 \cite{lh} and Lemma 2,10 \cite{y1} are still true when $f$ and $g$ share three value IM almost.

\begin{lemma}\label{13}  Let $f$ be a transcendental meromorphic function, let $k_{j}(j=1,2,\ldots,q)$ be  distinct constants, and let $a_{1}\not\equiv\infty$ and $a_{2}\not\equiv\infty$ be two distinct small functions of $f$ . Again let $d_{j}=a-k_{j}(a-b)$ $(j=1,2,\ldots,q)$. Then
$$m(r,\frac{L(f)}{f-a_{1}})=S(r,f), \quad m(r,\frac{L(f)}{f-d_{j}})=S(r,f).$$
for $1\leq i\leq q$ and 
$$m(r,\frac{L(f)f}{(f-d_{1})(f-d_{2})\cdots(f-d_{m})})=S(r,f),$$
where $L(f)=(a'_{1}-a'_{2})(f-a_{1})-(a_{1}-a_{2})(f'-a'_{1})$, and $2\leq m\leq q$.
\end{lemma}
\begin{proof}
We first claim that $L(f)\not\equiv0$. Otherwise, if $L(f)\equiv0$,  then we can get $\frac{f'-a'_{1}}{f-a_{1}}\equiv\frac{a'_{1}-a'_{2}}{a_{1}-a_{2}}$. Integrating both side of above we can obtain $f-a_{1}=C_{1}(a_{1}-a_{2})$, where $C_{1}$ is a nonzero constant. So by Lemma 2.3, we have $T(r,f)=T(r,f_{c})+S(r,f)=T(r,C(a_{1}-a_{2})+a_{1})=S(r,f)$, a contradiction. Hence $L(f)\not\equiv0$. Obviously, we have
$$m(r,\frac{L(f)}{f-a_{1}})\leq m(r,\frac{(a'-b')(f-a_{1})}{f-a_{1}})+m(r,\frac{(a-b)(f'-a'_{1})}{f-a_{1}})=S(r,f),$$
and
$$\frac{L(f)f}{(f-d_{1})(f-d_{2})\cdots(f-d_{q})}=\sum_{i=1}^{q}\frac{C_{i}L(f)}{f-d_{i}},$$
where $C_{i}=\frac{d_{j}}{\prod\limits_{j\neq i}(d_{i}-d_{j})}$ are small functions of $f$. By Lemma 2.1 and above, we have
\begin{align}
&m(r,\frac{L(f)f}{(f-d_{1})(f-d_{2})\cdots(f-d_{q})})=m(r,\sum_{i=1}^{q}\frac{C_{i}L(f)}{f-d_{i}})\notag\\
&\leq\sum_{i=1}^{q}m(r,\frac{L(f_{c})}{f_{c}-d_{i}})+S(r,f)=S(r,f).
\end{align}
\end{proof}

\section{The proof of Theorem 1}
If $f\equiv g$, there is nothing to prove. Suppose $f\not\equiv g$. Since $f$ is a non-constant meromorphic function,  $f$ and $g$ share $a_{1},\infty$ CM, then  we get
\begin{align}
\frac{g-a_{1}}{f-a_{1}}=q,
\end{align}
where $q$ is a meromorphic function, and (2.2) implies  $h=\frac{1}{q}$.\\

Since $f$ and $g$ share $a_{1},\infty$ CM and share $a_{2}$ IM, then  Lemma 2.1, Lemma 2.2 and Lemma 2.12 we have
\begin{eqnarray*}
\begin{aligned}
T(r,f)&\leq\overline{N}(r,f)+ \overline{N}(r,\frac{1}{f-a_{1}})+\overline{N}(r,\frac{1}{f-a_{2}})+S(r,f)= \overline{N}(r,\frac{1}{g-a_{1}})\\
&+\overline{N}(r,\frac{1}{g-b})+S(r,f)\leq N(r,\frac{1}{f-g})+S(r,f)\\
&\leq T(r,f-g)+S(r,f)\leq m(r,f-g)+S(r,f)\\
&\leq m(r,f-\sum_{i=0}^{n}b_{i}f^{(k_{i})}_{c_{i}})+S(r,f)\\
&\leq m(r,f)+m(r,1-\frac{\sum_{i=0}^{n}b_{i}f^{(k_{i})}_{c_{i}}}{f})+S(r,f)\\
&\leq T(r,f)+S(r,f).
\end{aligned}
\end{eqnarray*}

That is
\begin{eqnarray}
T(r,f)=\overline{N}(r,\frac{1}{f-a_{1}})+\overline{N}(r,\frac{1}{f-a_{2}})+S(r,f).
\end{eqnarray}
According to (3.1) and (3.2) we have
\begin{eqnarray}
T(r,f)=T(r,f-g)+S(r,f)=N(r,\frac{1}{f-g})+S(r,f).
\end{eqnarray}
and
\begin{align}
T(r,q)&=m(r,q)+S(r,f)\notag\\
&=m(r,\frac{g-a_{1}}{f-a_{1}})+S(r,f)\notag\\
&\leq m(r,\frac{1}{f-a_{1}})+S(r,f).
\end{align}
Then it follows from (3.1) and (3.3) that
\begin{align}
m(r,\frac{1}{f-a_{1}})&=m(r,\frac{q-1}{f-g})\notag\\
&\leq m(r,\frac{1}{f-g})+m(r,q-1)\notag\\
&\leq T(r,q)+S(r,f).
\end{align}
Then by (3.4) and (3.5)
\begin{align}
T(r,q)= m(r,\frac{1}{f-a_{1}})+S(r,f).
\end{align}
On the other hand, (3.1) can be rewritten as
\begin{align}
\frac{g-f}{f-a_{1}}=q-1,
\end{align}
which implies
\begin{align}
\overline{N}(r,\frac{1}{f-a_{2}})\leq \overline{N}(r,\frac{1}{q-1})=T(r,q)+S(r,f).
\end{align}
Thus, by (3.2), (3.6) and (3.8)
\begin{eqnarray*}
\begin{aligned}
m(r,\frac{1}{f-a_{1}})+N(r,\frac{1}{f-a_{1}})&= \overline{N}(r,\frac{1}{f-a_{1}})+\overline{N}(r,\frac{1}{f-a_{2}})+S(r,f)\\
&\leq \overline{N}(r,\frac{1}{f-a_{1}})+\overline{N}(r,\frac{1}{q-1})+S(r,f)\\
&\leq\overline{N}(r,\frac{1}{f-a_{1}})+m(r,\frac{1}{f-a_{1}})+S(r,f),
\end{aligned}
\end{eqnarray*}
that is
\begin{align}
N(r,\frac{1}{f-a_{1}})=\overline{N}(r,\frac{1}{f-a_{1}})+S(r,f).
\end{align}
And then
\begin{align}
\overline{N}(r,\frac{1}{f-a_{2}})=T(r,q)+S(r,f).
\end{align}
Set
\begin{eqnarray}
\varphi=\frac{f'(f-g)}{(f-a_{1})(f-a_{2})},
\end{eqnarray}
and
\begin{eqnarray}
\psi=\frac{g'(f-g)}{(g-a_{1})(g-a_{2})}.
\end{eqnarray}
 Easy to know that $\varphi\not\equiv0$ because of $f\not\equiv g $, and $N(r,\varphi)=S(r,f)$. By   Lemma 2.1 and Lemma 2.7 we have
\begin{eqnarray*}
\begin{aligned}
&T(r,\varphi)=m(r,\varphi)=m(r,\frac{f'(f-g)}{(f-a_{1})(f-a_{2})})+S(r,f)\notag\\
&=m(r,\frac{ff'}{(f-a_{1})(f-a_{2})}\frac{f-\sum_{i=0}^{n}b_{i}f^{(k_{i})}_{c_{i}}}{f})+m(r,\frac{b_{-1}ff'}{(f-a_{1})(f-a_{2})})+S(r,f)\\
&\leq m(r,\frac{ff'}{(f-a_{1})(f-a_{2})})+m(r,1-\frac{\sum_{i=0}^{n}b_{i}f^{(k_{i})})_{c_{i}}}{f})+S(r,f)=S(r,f),
\end{aligned}
\end{eqnarray*}
that is
\begin{align}
T(r,\varphi)=S(r,f).
\end{align}
Let $d=a_{1}-j(a_{1}-b_{1})(j\neq0,1)$. Obviously, by  Lemma 2.1, Lemma 2.12 and  the First Fundamental Theorem of Nevanlinna, we obtain
\begin{eqnarray*}
\begin{aligned}
 2T(r,f)&\leq \overline{N}(r,f)+\overline{N}(r,\frac{1}{f-a_{1}})+ \overline{N}(r,\frac{1}{f-a_{2}})+\overline{N}(r,\frac{1}{f}) +S(r,f)\\
 &\leq T(r,f)+T(r,f)-m(r,\frac{1}{f})+S(r,f),
\end{aligned}
\end{eqnarray*}
which implies
\begin{align}
m(r,\frac{1}{f})=S(r,f).
\end{align}
And by (3.14)
\begin{align}
m(r,\frac{1}{f-d})&=m(r,\frac{f'(f-g)}{\varphi (f-a_{1})(f-a_{2})(f-d)})\leq m(r,1-\frac{g}{f})\notag\\
&+m(r,\frac{ff'}{(f-a_{1})(f-a_{2})(f-d)})+S(r,f)=S(r,f).
\end{align}
Set
\begin{align}
\phi=\frac{g'}{(g-a_{1})(g-a_{2})}-\frac{f'}{(f-a_{1})(f-a_{2})}.
\end{align}
We  discuss  two cases.\\

{\bf Case 1}\quad $\phi\equiv0$.  Integrating the both side of (3.16) which leads to
\begin{align}
\frac{f-a_{2}}{f-a_{1}}=C\frac{g-a_{2}}{g-a_{1}},
\end{align}
where $C$ is a nonzero constant.
Then by Lemma 2.4 we get
\begin{eqnarray}
2T(r,f)\leq\overline{N}(r,\frac{1}{f-a_{1}})+\overline{N}(r,\frac{1}{f-a_{2}})+S(r,f),
\end{eqnarray}
which contradicts with (3.2).

{\bf Case 2} \quad $\phi \not\equiv0$. By (3.3), (3.13) and (3.16), we can obtain
\begin{align}
m(r,f)&=m(r,f-g)+S(r,f)\notag\\
&=m(r,\frac{\phi(f-g)}{\phi})+S(r,f)=m(r,\frac{\psi-\varphi}{\phi})+S(r,f)\notag\\
&\leq T(r,\frac{\phi}{\psi-\varphi})+S(r,f)\leq T(r,\psi-\varphi)+T(r,\phi)+S(r,f)\notag\\
&\leq T(r,\psi)+T(r,\phi)+S(r,f)\notag\\
&\leq T(r,\psi)+\overline{N}(r,\frac{1}{f-a_{2}})+S(r,f),
\end{align}
on the other hand,
\begin{align}
T(r,\psi)&=T(r,\frac{g'(f-g)}{(g-a_{1})(g-a_{2})})\notag\\
&=m(r,\frac{g'(f-g)}{(g-a_{1})(g-a_{2})})+S(r,f)\notag\\
&\leq m(r,\frac{g'}{g-a_{2}})+m(r,\frac{f-g}{g-a_{1}})\notag\\
&\leq m(r,\frac{1}{f-a_{1}})+S(r,f)=\overline{N}(r,\frac{1}{f-a_{2}})+S(r,f).
\end{align}
Hence combining  (3.19) and (3.20), we obtain
\begin{align}
 T(r,f)\leq 2\overline{N}(r,\frac{1}{f-a_{2}})+S(r,f).
\end{align}

Next, Case 2 is  divided into two subcases.

{\bf Subcase 2.1}\quad $a_{1}=G$, where $G$ is defined as (2.2) in Lemma 2.5. Then by (3.1) and Lemma 2.1 we can get
\begin{align}
 m(r,q)=m(r,\frac{g-G}{f-a_{1}})=S(r,f).
\end{align}
Then by (3.10), (3.21)  and (3.22) we can have $T(r,f)=S(r,f)$, a contradiction.\\

{\bf Subcase 2.2} \quad $a_{2}=G$. Then by (3.6), (3.10) and (3.21), we get
\begin{align}
 T(r,f)&\leq m(r,\frac{1}{f-a_{1}})+\overline{N}(r,\frac{1}{g-G})+S(r,f)\notag\\
 &\leq m(r,\frac{1}{g-G})+\overline{N}(r,\frac{1}{g-G})+S(r,f)\notag\\
 &\leq T(r,g)+S(r,f).
\end{align}
From the fact that
\begin{align}
 T(r,g)\leq T(r,f)+S(r,f),
\end{align}
which follows from (3.23) that
\begin{align}
 T(r,f)=T(r,g)+S(r,f).
\end{align}
By Lemma 2.2, (3.2) and (3.25), we have
\begin{eqnarray*}
\begin{aligned}
2T(r,f)&\leq 2T(r,g)+S(r,f)\\
&\leq\overline{N}(r,g)+\overline{N}(r,\frac{1}{g-a_{1}})+\overline{N}(r,\frac{1}{g-G})+\overline{N}(r,\frac{1}{g-d})+S(r,f)\\
&\leq \overline{N}(r,\frac{1}{f-a_{1}})+\overline{N}(r,\frac{1}{f-a_{2}})+T(r,\frac{1}{g-d})-m(r,\frac{1}{g-d})+S(r,f)\\
&\leq T(r,f)+T(r,g)-m(r,\frac{1}{g-d})+S(r,f)\\
&\leq 2T(r,f)-m(r,\frac{1}{g-d})+S(r,f).
\end{aligned}
\end{eqnarray*}
Thus
\begin{eqnarray}
m(r,\frac{1}{g-d})=S(r,f).
\end{eqnarray}

From the First Fundamental Theorem,  Lemma 2.1, (3.14)-(3.15), (3.25)-(3.26) and  $f$ is a non-constant  meromorphic function, we obtain
\begin{eqnarray*}
\begin{aligned}
m(r,\frac{f-d}{g-d})&\leq m(r,\frac{f}{g-d})+m(r,\frac{d}{g-d})+S(r,f)\\
&\leq T(r,\frac{f}{g-d})-N(r,\frac{f}{g-d})+S(r,f)\\
&=m(r,\frac{g-d}{f})+N(r,\frac{g-d}{f})-N(r,\frac{f}{g-d})\\
&+S(r,f)\leq N(r,\frac{1}{f})-N(r,\frac{1}{g-d})+S(r,f)\\
&=T(r,\frac{1}{f})-T(r,\frac{1}{g-d})+S(r,f)\\
&=T(r,f)-T(r,g)+S(r,f)=S(r,f).
\end{aligned}
\end{eqnarray*}

Thus we get
\begin{eqnarray}
m(r,\frac{f-d}{g-d})=S(r,f).
\end{eqnarray}
It's easy to see that $N(r,\psi)=S(r,f)$  and (3.12) can be rewritten as
\begin{eqnarray}
\psi=[\frac{a_{1}-d}{a_{1}-a_{2}}\frac{g'}{g-a_{1}}-\frac{a_{2}-d}{a_{1}-a_{2}}\frac{g'}{g-a_{2}}][\frac{f-d}{g-d}-1].
\end{eqnarray}
Then by  (3.27) and (3.28) we can get
\begin{eqnarray}
T(r,\psi)=m(r,\psi)+N(r,\psi)=S(r,f).
\end{eqnarray}
By (3.2), (3.19), and (3.29) we get
\begin{eqnarray}
\overline{N}(r,\frac{1}{f-a_{1}})=S(r,f).
\end{eqnarray}
Moreover, by (3.2), (3.25) and (3.30), we have
\begin{eqnarray}
m(r,\frac{1}{g-G})=S(r,f),
\end{eqnarray}
which implies
\begin{eqnarray}
\overline{N}(r,\frac{1}{f-a_{2}})=m(r,\frac{1}{f-a_{2}})\leq m(r,\frac{1}{g-G})=S(r,f).
\end{eqnarray}
Then by (3.2) we obtain $T(r,f)=S(r,f)$, a contradiction.\\

{\bf Subcase 2.3} $a_{1}\not\equiv G, a_{2}\not\equiv G$. So by (3.6), (3.10), (3.21) and Lemma 2.1, we can get
\begin{eqnarray*}
\begin{aligned}
T(r,f)&\leq 2m(r,\frac{1}{f-a_{1}})+S(r,f)\leq2m(r,\frac{1}{g-G})\\
&+S(r,f)=2T(r,g)-2N(r,\frac{1}{g-G})+S(r,f)\\
&\leq\overline{N}(r,g)+\overline{N}(r,\frac{1}{g-a_{1}})+\overline{N}(r,\frac{1}{g-a_{2}})+\overline{N}(r,\frac{1}{g-G})\\
&-2N(r,\frac{1}{g-G})+S(r,f)\leq T(r,f)-N(r,\frac{1}{g-G})+S(r,f),
\end{aligned}
\end{eqnarray*}
which deduces that
\begin{align}
N(r,\frac{1}{g-G})=S(r,f).
\end{align}
It follows from (3.33) and Lemma 2.12 that
\begin{eqnarray*}
\begin{aligned}
T(r,g)&\leq \overline{N}(r,g)+\overline{N}(r,\frac{1}{g-G})+\overline{N}(r,\frac{1}{g-a_{1}})+S(r,f)\\
&\leq \overline{N}(r,\frac{1}{g-a_{1}})+S(r,f)\leq T(r,g)+S(r,f),
\end{aligned}
\end{eqnarray*}
which implies that
\begin{align}
T(r,g)=\overline{N}(r,\frac{1}{g-a_{1}})+S(r,f).
\end{align}
Similarly
\begin{align}
T(r,g)=\overline{N}(r,\frac{1}{g-a_{2}})+S(r,f).
\end{align}
Then by (3.21) we get
\begin{align}
T(r,f)=2T(r,g)+S(r,f).
\end{align}
Easy to see from (3.16) that
\begin{align}
T(r,\phi)=N(r,\phi)+S(r,f)\leq\overline{N}(r,\frac{1}{g-a_{2}})+S(r,f).
\end{align}
We claim that
\begin{align}
T(r,\phi)=\overline{N}(r,\frac{1}{g-a_{2}})+S(r,f).
\end{align}
Otherwise, 
\begin{align}
T(r,\phi)<\overline{N}(r,\frac{1}{g-a_{2}})+S(r,f).
\end{align}
We can  deduce from (3.2), (3.12), Lemma 2.1 and Lemma 2.3  that
\begin{eqnarray*}
\begin{aligned}
T(r,\psi)&=T(r,\frac{g'(f-g)}{(g-a_{1})(g-a_{2})})=m(r,\frac{g'(f-g)}{(g-a_{1})(g-a_{2})})+S(r,f)\notag\\
&\leq m(r,\frac{g'}{g-a_{1}})+m(r,\frac{f-a_{2}}{g-a_{2}}-1)\notag\\
&\leq m(r,\frac{g-a_{2}}{f-a_{2}})+N(r,\frac{g-a_{2}}{f-a_{2}})-N(r,\frac{f-a_{2}}{g-a_{2}})+S(r,f)\\
&\leq m(r,\frac{1}{f-a_{2}})+N(r,\frac{1}{f-a_{2}})-N(r,\frac{1}{g-a_{2}})+S(r,f)\\
&\leq T(r,f)-\overline{N}(r,\frac{1}{g-a_{2}})+S(r,f)\leq \overline{N}(r,\frac{1}{f-a_{1}})+S(r,f),
\end{aligned}
\end{eqnarray*}
which is
\begin{align}
T(r,\psi)\leq \overline{N}(r,\frac{1}{f-a_{1}})+S(r,f).
\end{align}
Then combining (3.2), (3.39)-(3.40) and the proof of (3.19), we obtain
\begin{eqnarray*}
\begin{aligned}
&\overline{N}(r,\frac{1}{f-a_{1}})+\overline{N}(r,\frac{1}{f-a_{2}})=T(r,f)+S(r,f)\notag\\
&\leq \overline{N}(r,\frac{1}{f-a_{1}})+T(r,\phi)+S(r,f),
\end{aligned}
\end{eqnarray*}
that is
\begin{align}
\overline{N}(r,\frac{1}{g-a_{2}})\leq T(r,\phi)+S(r,f),
\end{align}
a contradiction. Similarly, we can also obtain
\begin{align}
T(r,\psi)=\overline{N}(r,\frac{1}{g-a_{1}})+S(r,f).
\end{align}

By Lemma 2.5, if
\begin{align}
g=Hh+G,
\end{align}
where $H\not\equiv0$ is a small function of $h$.\\

Rewrite (3.16) as
\begin{align}
\phi\equiv\frac{g'(f-a)(f-b)-f'(g-a_{1})(g-a_{2})}{(f-a_{1})(f-a_{2})(g-a_{1})(g-a_{2})}.
\end{align}
Set $a=\frac{h'}{h}$. Since $N(r,h)+N(r,\frac{1}{h})=S(r,f)$, we obtain from Lemma 2.1 that
$$T(r,a)=m(r,\frac{h'}{h})+N(r,\frac{h'}{h})=S(r,f),$$
which implies that $a$ is a small function of $f$.

Combing (2.1) with (3.44),  we can set
\begin{align}
P&=g'(f-a)(f-b)-f'(g-a_{1})(g-a_{2})\notag\\
&=\sum_{i=0}^{5}\alpha_{i}h^{i},
\end{align}
and
\begin{align}
Q&=(f-a_{1})(f-a_{2})(g-a_{1})(g-a_{2})\notag\\
&=\sum_{j=0}^{6}\beta_{j}h^{j},
\end{align}
where $\alpha_{i}$ and $\beta_{j}$ are small functions of $h$, and $\alpha_{5}\not\equiv0$, $\beta_{6}\not\equiv0$.

If $P$ and $Q$ are two mutually prime polynomials in $e^{p}$, then by Lemma 2.9 we can get $T(r,\phi)=6T(r,h)+S(r,f)$. It follows from (3.10), (3.45)-(3.46) that $T(r,f)=S(r,f)$, a contradiction.\\

If $P$ and $Q$ are  not two mutually prime polynomials in $h$, it's easy to see that the degree of $Q$ is large than $P$.\\
According to (3.38), (3.45), (3.46) and by simple calculation,  we must have
\begin{align}
\phi=\frac{C}{g-a_{2}},
\end{align}
where $C\not\equiv0$ is a small function of $f$.\\
Put (3.47) into (3.16) we have
\begin{align}
\frac{C(g-a_{1})-g'}{(g-a_{1})(g-a_{2})}\equiv\frac{-f'}{(f-a_{1})(f-a_{2})}.
\end{align}
By (3.43) and (3.48), we claim that $CH\equiv DH+aH$. Otherwise, combining (3.16), (3.38),(3.43) and Lemma 2.8, we can get $T(r,h)=S(r,f)$. It follows from  (3.10) and (3.21) that $T(r,f)=S(r,f)$, a contradiction.  Then by (3.1), (3.48) and $CH\equiv H'+aH$, we have
\begin{align}
\frac{G'-C(G-a_{1})}{Hh+G-a_{2}}\equiv\frac{(2aH+H')h+a(G-a_{1})+G'}{h(Hh+G-a_{1})+a_{1}-a_{2}}.
\end{align}
From above equality and $CH\equiv H'+aH$, we obtain the followings equalities.
\begin{align}
A\equiv(a+C)H,
\end{align}
\begin{align}
[a(G-a_{1})+G'](G-a_{2})\equiv A(a_{1}-a_{2}),
\end{align}
and
\begin{align}
a(G-a_{1})+G'+(a+C)(G-a_{2})\equiv (a+C)(G-a_{1}),
\end{align}
where $A\equiv G'-C(G-a_{1})$. By (3.50)-(3.52) we have
\begin{align}
a_{2}\equiv G+H.
\end{align}
Differential above we get
\begin{align}
(H+G)' \equiv 0,
\end{align}
which implies
\begin{align}
(C-a)H+(a+C)H+C(G-a_{1})\equiv C(a_{2}+H-a_{1})\equiv0.
\end{align}
Therefore, we can see from (3.53) and (3.55) that
\begin{align}
G=2a_{2}-a_{1},
\end{align}
it follows from (3.53) that
\begin{align}
H=a_{1}-a_{2}.
\end{align}
Combining (3.43), (3.56) and (3.57), we have
\begin{align}
g=(a_{1}-a_{2})h+2a_{2}-a_{1}.
\end{align}
And then by (2.1) we have
\begin{align}
f=a_{2}+(a_{1}-a_{2})(h-1)^{2}.
\end{align}

If $m(r,h)=S(r,f)$, then by (3.10) and (3.21), we deduce $T(r,f)=S(r,f)$, a contradiction.\\

This completes the proof of Theorem 1.

\section{The proof of Corollary 1}
Assume that $f\equiv g$. Set
$F=\frac{f-a_{1}}{a_{2}-a_{1}}$ and $G=\frac{g-a_{1}}{a_{2}-a_{1}}$. We know that $F$ and $G$ share $0$ CM almost and $1$ IM almost. Obviously, we know that $G$ is still a differential-difference polynomial in $F$. Then by Theorem 1 and Remark 1, we have
\begin{align}
G=(a_{1}-a_{2})h+2a_{2}-a_{1},
\end{align}
and
\begin{align}
F=a_{2}+(a_{1}-a_{2})(h-1)^{2}.
\end{align}
Therefore, if $g=f^{(k)}(z+c)$,   since  $f$ is a transcendental meromorphic function with $\rho_{2}(f)<1$ and $f^{(k)}_{c}$ and $f$ share  $ \infty$ CM, we can see from  Lemma 2.1 and Lemma 2.3 that
\begin{eqnarray*}
\begin{aligned}
(1+o(1))N(r,f)+S(r,f)=N(r,f_{c})=N(r,f^{(k)}_{c}),
\end{aligned}
\end{eqnarray*}
and on the other hand
\begin{eqnarray*}
\begin{aligned}
k\overline{N}(r,f_{c})+N(r,f_{c})=N(r,f^{(k)}_{c}),  \overline{N}(r,f_{c})=\overline{N}(r,f^{(k)})=\overline{N}(r,f),
\end{aligned}
\end{eqnarray*} 
which follows from above equalities that $\overline{N}(r,f)=S(r,f)$. By Theorem 1, we have 
\begin{align}
f=a_{2}+(a_{1}-a_{2})(e^{p}-1)^{2},
\end{align}
and 
\begin{align}
f^{(k)}_{c}=(a_{1}-a_{2})e^{p}+2a_{2}-a_{1},
\end{align}
where the reason $h=e^{p}$ with $p$ a non-constant polynomial is that $f^{(k)}_{c}$ and $f$ share  $a_{1}, \infty$ CM, and by (2.1) we know that $h$ is an entire function.

By (4.3) we have
\begin{align}
f^{(k)}_{c}=P(z)e^{2p_{c}}+Q(z)e^{p_{c}}+a_{2}^{(k)},
\end{align}
where $P(z)\not\equiv0$ and $Q(z)\not\equiv0$ are differential polynomial in $$A_{c}, A'_{c},\ldots,A^{(k)}, p_{c},p'_{c},\ldots,p^{(k)},$$
and $A=a_{1}-a_{2}$. So we can not obtain (4.4) from (4.3).

And hence from above discussions, we only obtain $f\equiv f^{(k)}_{c}$.

\section{The Proof of Theorem 2}
 If $f(z)\equiv f(z+c)$, there is nothing to do. Assume that $f(z)\not\equiv f(z+c)$. Since $f(z)$ is a transcendental meromorphic function of $\rho_{2}(f)<1$, $f$ and $f(z+c)$ share $a_{1}(z),\infty$ CM, then there is a nonzero entire function $p(z)$  of order less than $1$ such that
\begin{eqnarray}
\frac{f(z+c)-a_{1}(z)}{f(z)-a_{1}(z)}=e^{p(z)},
\end{eqnarray}
then by Lemma 2.1 and $a(z)$ is a periodic function with period $c$,
\begin{eqnarray}
T(r,e^{p})=m(r,e^{p})=m(r,\frac{f(z+c)-a_{1}(z+c)}{f(z)-a_{1}(z)})=S(r,f).
\end{eqnarray}
On the other hand, (4.1) can be rewritten as
\begin{eqnarray}
\frac{f(z+c)-f(z)}{f(z)-a_{1}(z)}=e^{p(z)}-1.
\end{eqnarray}
We put
\begin{align}
\varphi(z)=\frac{L(f)(f(z+c)-f(z))}{(f(z)-a_{1}(z))(f(z)-a_{2}(z))},
\end{align}
where $L(f)$ is defined as in Lemma 2.13. Then by Lemma 2.1, Lemma 2.13 and the fact that $f(z)$ and $f(z+c)$ share $a_{1}(z)$ and $\infty$ CM, and $a_{2}(z)$ IM, we have
\begin{eqnarray*}
\begin{aligned}
T(r,\varphi(z))&=m(r,\varphi(z))+N(r,\varphi(z))\\
&\leq m(r,\frac{L(f)f(z)}{(f(z)-a_{1}(z))(f(z)-a_{2}(z))})+m(r,\frac{f(z+c)-f(z)}{f(z)})\\
&+\overline{N}(r,f(z))+S(r,f)\leq \overline{N}(r,f(z))+S(r,f),
\end{aligned}
\end{eqnarray*}
which implies $T(r,\varphi(z))\leq \overline{N}(r,f(z))+S(r,f)$. If $T(r,\varphi(z))< \overline{N}(r,f(z))+S(r,f)$, then we can know from the fact  $f(z)$ and $f(z+c)$ share  $\infty$ CM and the definition of $\varphi(z)$ that $\overline{N}(r,f(z))=S(r,f)$. Then we can see from Theorem G and Remark 1 that $f(z)\equiv f(z+c)$. Hence
\begin{align}
T(r,\varphi(z))= \overline{N}(r,f(z))+S(r,f).
\end{align}
Because the zeros of $f(z+c)-f(z)$ but not the zeros of $f(z)-a(z)$ nor $f(z)-b(z)$ are the zeros of $e^{p(z)}-1$, and hence by (5.4) we have
\begin{align}
N(r,\frac{1}{\varphi(z)})&= N(r,\frac{1}{L(f)})+N_{(2}(r,\frac{1}{f-a_{1}})+N_{(m,n)}(r,\frac{1}{f-a_{2}})\notag\\
&-\overline{N}_{(2}(r,\frac{1}{f-a_{1}})-\overline{N}_{(m,n)}(r,\frac{1}{f-a_{2}})+S(r,f).
\end{align}
On the other hand, we know from (5.1)-(5.4), Nevanlinna's first fundamental theorem, Lemma 2.2 and Lemma 2.13 that
\begin{eqnarray*}
\begin{aligned}
m(r,\frac{1}{\varphi(z)})&= m(r,\frac{1}{\varphi(z)(e^{p(z)}-1)})+S(r,f)=m(r,\frac{f-a_{2}}{L(f)})+S(r,f)\\
&=m(r,\frac{L(f)}{f-b})+N(r,\frac{L(f)}{f-a_{2}})-N(r,\frac{f-a_{2}}{L(f)})+S(r,f)\\
&=N(r,L(f))+N(r,\frac{1}{f-a_{2}})-N(r,f)-N(r,\frac{1}{L(f)})+S(r,f)\\
&=\overline{N}(r,f)+N(r,\frac{1}{f-a_{2}})-N(r,\frac{1}{L(f)})+S(r,f),
\end{aligned}
\end{eqnarray*}
the first equality holds because $T(r,e^{p})=S(r,f)$, so it follows from above that
\begin{align}
m(r,\frac{1}{\varphi(z)})=\overline{N}(r,f)+N(r,\frac{1}{f-a_{2}})-N(r,\frac{1}{L(f)})+S(r,f).
\end{align}
It is easy to see from  (5.5)-(5.7) that
\begin{eqnarray*}
\begin{aligned}
&\overline{N}(r,f(z))=T(r,\varphi(z))+S(r,f)=m(r,\frac{1}{\varphi(z)})+N(r,\frac{1}{\varphi(z)})+S(r,f)\\
&=\overline{N}(r,f(z))+N(r,\frac{1}{f(z)-a_{2}(z)})-N(r,\frac{1}{L(f)})+N(r,\frac{1}{L(f)})\\
&+N_{(2}(r,\frac{1}{f(z)-a_{1}(z)})+N_{(m,n)}(r,\frac{1}{f(z)-a_{2}(z)})-\overline{N}_{(2}(r,\frac{1}{f(z)-a_{1}(z)})\\
&-\overline{N}_{(m,n)}(r,\frac{1}{f(z)-a_{2}(z)})+S(r,f)=\overline{N}(r,f(z))+N(r,\frac{1}{f(z)-a_{2}(z)})+N_{(2}(r,\frac{1}{f(z)-a_{1}(z)})\\
&+N_{(m,n)}(r,\frac{1}{f(z)-a_{2}(z)})-\overline{N}_{(2}(r,\frac{1}{f(z)-a_{1}(z)})-\overline{N}_{(m,n)}(r,\frac{1}{f(z)-a_{2}(z)})+S(r,f),
\end{aligned}
\end{eqnarray*}
which implies
\begin{align}
N_{(2}(r,\frac{1}{f(z)-a_{1}(z)})+N(r,\frac{1}{f(z)-a_{2}(z)})=S(r,f).
\end{align}
We can know from (5.8), Lemma 2.1 and Lemma 2.3 that
\begin{align}
N(r,\frac{1}{f(z+c)-a_{2}(z+c)})=N(r,\frac{1}{f(z)-a_{2}(z)})+S(r,f)=S(r,f).
\end{align}

Set
\begin{eqnarray}
\psi(z)=\frac{f(z+c)-a_{2}(z+c)}{f(z)-a_{2}(z)}.
\end{eqnarray}
It is easy to see that
\begin{eqnarray}
N(r,\frac{1}{\psi(z)})\leq N(r,\frac{1}{f(z+c)-a_{2}(z+c)})+N(r,a_{2}(z))= S(r,f),
\end{eqnarray}
\begin{eqnarray}
N(r,\psi(z))\leq N(r,\frac{1}{f(z)-a_{2}(z)})+N(r,a_{2}(z+c))= S(r,f).
\end{eqnarray}
Hence by Lemma 2.1,
\begin{align}
T(r,\psi(z))&=m(r,\psi(z))+N(r,\psi(z))\notag\\
&\leq m(r,\frac{f(z+c)-a_{2}(z+c)}{f(z)-a_{2}(z)})+N(r,\frac{1}{f(z)-a_{2}(z)})\notag\\
&+S(r,f)\leq S(r,f).
\end{align}
Subtracting (5.10) from (5.1), we have
\begin{eqnarray}
(e^{p(z)}-\psi(z))f(z)+\psi(z)a_{2}(z)+a_{1}(z)-a_{2}(z+c)-a_{1}(z)e^{p(z)}\equiv0.
\end{eqnarray}
We discuss following two cases.\\

{\bf Case 1} \quad $e^{p(z)}\not\equiv\psi(z)$. Then by (4.2), (4.10) and (4.13) we obtain $T(r,f)=S(r,f)$, a contradiction.\\

{\bf Case 2} \quad $e^{p(z)}\equiv\psi(z)$. Then by (5.1) we have
\begin{eqnarray}
f(z+c)=e^{p(z)}(f(z)-a_{1}(z))+a_{1}(z),
\end{eqnarray}
and
\begin{eqnarray}
N(r,\frac{1}{f(z+c)-a_{2}(z)})=N(r,\frac{1}{f(z)-a_{1}(z)+\frac{a_{1}(z)-a_{2}(z)}{e^{p(z)}}})=S(r,f).
\end{eqnarray}
If $a_{2}(z)$ is a periodic function of period $c$, then by (5.11) we can get $e^{p(z)}\equiv1$, which implies $f(z)\equiv f(z+c)$, a contradiction. Obviously, $a_{1}(z)-\frac{a_{1}(z)-a_{2}(z)}{e^{p(z)}}\not\equiv a_{1}(z)$. Otherwise, we can deduce $a_{1}(z)\equiv a_{2}(z)$, a contradiction.\\

Next, we discuss three Subcases.

{\bf Subcase 2.1}\quad $a_{1}(z)-\frac{a_{1}(z)-a_{2}(z)}{e^{p(z)}}\not\equiv a_{2}(z)$ and $a_{1}(z)-\frac{a_{1}(z)-a_{2}(z)}{e^{p(z)}}\not\equiv a_{2}(z-c)$. Then according to (5.5), (5.6),(5.16) and Lemma 2.12, we can get $T(r,f)=S(r,f)$, a contradiction.\\

{\bf Subcase 2.2}\quad $a_{1}(z)-\frac{a_{1}(z)-a_{2}(z)}{e^{p(z)}}\equiv a_{2}(z)$, but $a_{1}(z)-\frac{a_{1}(z)-a_{2}(z)}{e^{p(z)}}\not\equiv a_{2}(z-c)$. It follows that $e^{p(z)}\equiv1$. Therefore by (5.1) we have $f(z)\equiv f(z+c)$, a contradiction.

{\bf Subcase 2.3}\quad $a_{1}(z)-\frac{a_{1}(z)-a_{2}(z)}{e^{p(z)}}\not\equiv a_{2}(z)$ and $a_{1}(z)-\frac{a_{1}(z)-a_{2}(z)}{e^{p(z)}}\equiv a_{2}(z-c)$. It is easy to see that
\begin{eqnarray}
\frac{a_{1}(z)-a_{2}(z)}{a_{1}(z-c)-a_{2}(z-c)}=e^{p(z)}.
\end{eqnarray}
Furthermore, (5.14) implies
\begin{eqnarray}
\frac{a_{1}(z+c)-a_{2}(z+c)}{a_{1}(z)-a_{2}(z)}=e^{p(z)},
\end{eqnarray}
\begin{eqnarray}
\frac{a_{1}(z)-a_{2}(z)}{a_{1}(z-c)-a_{2}(z-c)}=e^{p(z-c)}.
\end{eqnarray}
It follows from (5.17) and (5.19) that
\begin{eqnarray}
e^{p(z)}=e^{p(z+c)}.
\end{eqnarray}
We also set
\begin{align}
F(z)=\frac{f(z)-a_{1}(z)}{a_{2}(z)-a_{1}(z)}, \quad G(z)=\frac{f(z+c)-a_{1}(z)}{a_{2}(z)-a_{1}(z)}.
\end{align}
Since $f(z)$ and $f(z+c)$ share $a_{1}(z)$ and $\infty$ CM, and $a_{2}(z)$ IM, so $F(z)$ and $G(z)$ share $0,\infty$ CM almost, and $1$ IM almost. We claim that $F$ is not a bilinear transform of $G$. Otherwise, we can see from Lemma 2.9 that if (i) occurs, we have $F(z)\equiv G(z)$, that is $f(z)\equiv f(z+c)$. If (ii) occurs, we have 
\begin{align}
N(r,\frac{1}{f(z)-a_{1}(z)})=S(r,f ), \quad N(r,f(z))=S(r,f).
\end{align}
Then by (5.8), (5.22) and Lemma 2.13, we can get  $T(r,f)=S(r,f)$, a contradiction.\\
If (iii) occurs, we have
\begin{align}
N(r,\frac{1}{f(z)-a_{1}(z)})=S(r,f ), \quad N(r,\frac{1}{f(z)-a_{2}(z)})=S(r,f).
\end{align}
Then it follows from above, $a_{1}(z)-\frac{a_{1}(z)-a_{2}(z)}{e^{p(z)}}\not\equiv a_{1}(z)$, $a_{1}(z)-\frac{a_{1}(z)-a_{2}(z)}{e^{p(z)}}\not\equiv a_{2}(z)$ and Lemma 2.13 that  $T(r,f)=S(r,f)$, a contradiction.\\
If (iv) occurs, we have $F(z)\equiv jG(z)$, that is 
\begin{align}
f(z)-a_{1}(z)=j(f(z+c)-a_{1}(z)),
\end{align}
where $j\neq0,1$ is a finite constant. By (5.1) and (5.24), we have $e^{p(z)}\equiv j$.  If $a_{2}(z+c)\not\equiv a_{2}(z-c)$, then by (5.9), (5.16) and Lemma 2.13, we get $T(r,f)=S(r,f)$, a contradiction. Thus, we have $a_{2}(z+c)\equiv a_{2}(z-c)$. Moreover, since $a(z)$ is a periodic function with period $c$, we can deduce $j^{2}=1$ from $e^{p(z)}\equiv j$ and (5.18)-(5.19). If $j=1$, we have $f(z+c)\equiv f(z)$, a contradiction. If $j=-1$, we obtain $f(z+c)=2a_{1}(z)-f(z)$. Then by Lemma 2.10, we know that either $f(z)\equiv f(z+c)$, a contradiction. Or $N(r,f(z))=S(r,f)$, and in this case  we can see from Theorem G and Remark 1 that $f(z)\equiv f(z+c)$, a contradiction.\\

If (v) occurs, we have 
\begin{align}
N(r,\frac{1}{f(z)-a_{1}(z)})=S(r,f ).
\end{align}
Then by Lemma 2.12, (5.9), (5.16) and $a_{2}(z-c)\not\equiv a_{1}(z)$, we obtain $T(r,f)=S(r,f)$, a contradiction.\\
If (vi) occurs, we have 
\begin{align}
 N(r,f(z))=S(r,f).
\end{align}
And hence  we can see from Theorem G and Remark 1  that $f(z)\equiv f(z+c)$, a contradiction.\\

Therefore, $F(z)$ is not a linear fraction transformation of $G(z)$.   If $a_{2}(z)$ is a small function with period $2c$, that is $a_{2}(z+c)\equiv a_{2}(z-c)$, we can set
\begin{align}
&D(z)=(f(z)-a_{2}(z))(a_{2}(z)-a_{2}(z-c))-(f(z+c)-a_{2}(z+c))(a_{2}(z+c)-a_{2}(z))\notag\\
&=(f(z)-a_{2}(z-c))(a_{2}(z)-a_{2}(z-c))-(f(z+c)-a_{2}(z))(a_{2}(z+c)-a_{2}(z))
\end{align}
If $D(z)\equiv0$, then we have $f(z+c)\equiv a_{2}(z+c)+a_{2}(z)-f(z)$, that is to say $F(z)$ is  a linear fraction transformation of $G(z)$, a contradiction. Hence $D(z)\not\equiv0$, and by (5.9)-(5.10), (5.16) and Lemma 2.1, we have
\begin{align}
2T(r,f(z))&=m(r,\frac{1}{f(z)-a_{2}(z)})+m(r,\frac{1}{f(z)-a_{2}(z-c)})+S(r,f)\notag\\
&=m(r,\frac{1}{f(z)-a_{2}(z)}+\frac{1}{f(z)-a_{2}(z-c)})+S(r,f)\notag\\
&\leq m(r,\frac{D(z)}{f(z)-a_{2}(z)}+\frac{D(z)}{f(z)-a_{2}(z-c)})+m(r,\frac{1}{D(z)})+S(r,f)\notag\\
&\leq m(r,\frac{1}{(\psi (z)+1)(f(z)-a_{2}(z))(a_{2}(z)-a_{2}(z-c))})+S(r,f)\notag\\
&\leq m(r,\frac{1}{f(z)-a_{2}(z)})+S(r,f)\leq T(r,f)+S(r,f),
\end{align}
which implies $T(r,f)=S(r,f)$, a contradiction.

By (5.18) we have
\begin{align}
\frac{\Delta_{c}a_{2}(z)}{1-e^{p(z)}}+a_{2}(z)=a_{1}(z).
\end{align}
Combining (5.20) and the fact that $a(z)$ is a small function with period $c$, we can get
\begin{align}
\frac{\Delta_{c}a_{2}(z+c)}{1-e^{p(z)}}+a_{2}(z+c)=a_{1}(z).
\end{align}
According to (5.29) and (5.30), we obtain
\begin{align}
e^{p(z)}=\frac{a_{2}(z+2c)-a_{2}(z+c)}{\Delta_{c}^{2}a_{2}(z)}.
\end{align}
So if $\rho(b(z))<\rho(e^{p(z)})$, we can follows from (5.31) and Lemma 2.11 that
\begin{align}
\rho(e^{p(z)})=\rho(\frac{a_{2}(z+2c)-a_{2}(z+c)}{\Delta_{c}^{2}a_{2}(z)})\leq \rho(a_{2}(z))<\rho(e^{p(z)}),
\end{align}
which is a contradiction.

If If $\rho(a_{2}(z))<1$, we claim that $p(z)\equiv B$ is a non-zero constant. Otherwise, the order of right hand side of (5.31)is $0$, but the left hand side is $1$, which is impossible. Therefore, by (5.1) we know that $f(z+c)-a_{1}(z)=B(f(z)-a_{1}(z))$, that is to say $F(z)$ is a linear fraction transformation of $G(z)$, a contradiction.

This completes Theorem 2.

\

{\bf Acknowledgements} The author would like to thank to  referee  for his helpful comments and also to the previous referee for giving the Example 4. The author also would like to express his grateful to Tohge Kazuya for communication about this paper.

\end{document}